\documentclass[11pt,draft]{amsart}
\usepackage{amssymb}
\usepackage{amsfonts}
\usepackage{tikz}
\usepackage{amsmath}

\theoremstyle{plain}

\newtheorem{theorem}{Theorem}

\newtheorem{corollary}[theorem]{Corollary}

\newtheorem{remark}{Remark}

\newcommand\tr{\operatorname{Tr}}

\begin{document}

\title[]{On isometry groups of self-adjoint traceless and skew-symmetric matrices}
\author{Marcell Ga\'al and Robert M. Guralnick}
\address{Department of Analysis, Bolyai Institute\\
         University of Szeged\\
         H-6720 Szeged, Aradi v\'ertan\'uk tere 1., Hungary}
\email{marcell.gaal.91@gmail.com}
\address{Department of Mathematics, University of Southern California, Los Angeles, CA 90089-2532
}
\email{guralnic@usc.edu}


\begin{abstract}
This note is concerned with isometries on the spaces of self-adjoint traceless matrices. We compute the group of isometries with respect to any unitary similarity invariant norm. This completes and extends the result of Nagy on Schatten $p$-norm isometries. Furthermore, we point out that our proof techniques could be applied to obtain an old result concerning isometries on skew-symmetric matrices.
\end{abstract}
\maketitle


\section{Introduction}

Let us denote $H^0_{n}$ the real vector space of $n$ by $n$ self-adjoint traceless matrices which has dimension $n^2-1$. To exclude trivial considerations we always assume that $n\geq 2$. Let $PSU(n)$ denote the image of $SU(n)$ in $GL(n^2-1,\mathbb{R})$ acting on $H^0_{n}$ via the adjoint representation
$$Ad:SU(n)\to GL(n^2-1,\mathbb{R}),\quad A\mapsto UAU^{-1}$$
which is isomorphic to $SU(n)/\{\zeta I\}$ where $\zeta$ runs through the set of $n$-th roots of unity and $I$ is the identity matrix. The symbol $\sigma$ stands for the Cartan involution which sends any element of $H^0_{n}$ to its negative transpose.

In the paper \cite{nagy} titled \textit{"Isometries of the spaces of self-adjoint traceless operators"} Nagy by means of "the invariance of domain" proved that the isometries on $H^0_{n}$ are automatically surjective and thus linear up to a translation, by the Mazur-Ulam theorem. Furthermore, the complete description of the structure of linear isometries with respect to a large class of unitary similarity invariant norm, more precisely, with respect to any Schatten $p$-norm $\|.\|_p$ was given whenever $n\neq 3$. Assume for a moment that $n>2$. Then the result of Nagy could be reformulated as follows. If $\mathcal{K}$ is the isometry group of $\|.\|_p$ on $H^0_{n}$, then one of the following happens:
\begin{itemize}
\item[(i)] $p\neq 2$ and $\mathcal{K}$ is generated by $PSU(n)$, $\mathbb{Z}_2$ and the Cartan involution;
\item[(ii)] $p=2$ and $\mathcal{K}=O(n^2-1,\mathbb{R})$.
\end{itemize}
Here, we consider the group $\mathbb{Z}_2$ as operators acting on $H^0_{n}$ by scalar multiplications with modulus one.
It was also pointed out in \cite{nagy} that the result remains true when any unitary invariant norm is considered in the case where $n=2$. In such a case both of the groups described above are the same.

In the present note we follow this line of research and determine the isometry group of any \textit{unitary invariant} norm on $H_{n}^0$. This problem was proposed in \cite{nagy} for further research. It turns out that we need to require only some weaker invariance property: the norm in question must be invariant under \textit{unitary similarity transformations}. At the same time we complete the former work of Nagy where the case $n=3$ has not been treated yet.

The clever proof of the main result in \cite{nagy} (i.e. the result on the corresponding isometry groups when $n\geq 3$) relies heavily on Uhlhorn's version of Wigner theorem. This theorem states that if $n\geq 3$, then every transformation between rank one orthogonal projections preserving the orthogonality is necessarily implemented by an either unitary or antiunitary similarity transformation.

Our method is based on the description of all compact $PSU(n)$ overgroups on self-adjoint traceless matrices. That is, we determine all the compact Lie groups lying between $PSU(n)$ and $GL(n^2-1,\mathbb{R})$. Then we select from the list which one preserve any unitary similarity invariant norm.
This approach for general algebraic group $G$ was first suggested by Dynkin in his classical paper for solving various linear preserver problems including certain $G$-invariant properties. Later, this scheme was employed effectively by the second named author in \cite{G1} to achieve results on invertible transformations on the full matrix algebra $M_n(\mathbb{F})$ (over any infinite field $\mathbb{F}$) preserving finite union of similarity classes.

Although the aforementioned classification result of Dynkin on maximal subgroups of algebraic groups are originally obtained for complex algebraic groups, by using some theory of semisimple Lie groups we are also able to determine compact overgroups. This has been done in \cite{guralnick} where among others compact $PSU(n)$ overgroups were determined on $M_n(\mathbb{C})$. Furthermore, the result was employed to achieve invertible preserver results on certain unitary similarity invariant quantities such as unitary similarity invariant norms, generalized numerical ranges and numerical radii. Consequently, our paper is a contribution to the previously mentioned results, as well.
As for further investigations in this direction, we mention the series of publications \cite{DL}, \cite{DP}, \cite{G1}, \cite{G2}, \cite{PD}.

Another motivation of the present note is to clarify matters concerning isometry groups of real skew-symmetric matrices. In 1944 the structure of spectral norm isometries was explored by Morita \cite{morita}. In 1991 a far-reaching generalization of Morita's result was given by Li and Tsing \cite{chi}. Namely, they described the general form of $c$-spectral norm isometries. Furthermore, at the end of the paper \cite{chi} the authors claimed that the problem has been already solved for orthogonal congruence invariant norms, as well. Nevertheless, any complete proof of the much more general statement has not been published yet. Fortunately, in this note we could present a short one.

Let us now turn to the formulation of the main result of the paper.
Clearly, $PSU(n)$ preserves the trace form $\langle A,B\rangle = \tr AB$ and this it embeds in the orthogonal group of $H_n^0$. Let us denote by $O(n^2-1,\mathbb{R})$ this group.
The main purpose of this note is to prove the following result.

\begin{theorem}\label{T:1}
Assume that $n\geq 3$. Let $\mathcal{K}$ be the isometry group of any unitary similarity invariant norm on $H^0_{n}$. Then one of the following holds:
\begin{itemize}
    \item [(i)] $\mathcal{K}=PSU(n)\rtimes\langle \mathbb{Z}_2, \sigma \rangle$;
    \item [(ii)] $\mathcal{K}=O(n^2-1,\mathbb{R})$.
\end{itemize}
\end{theorem}

We note that condition (i) of Theorem~\ref{T:1} holds if and only if the norm is not a multiple of the Frobenius norm and then the isometries have a simple structure. Further, condition (ii) is satisfied if and only if the norm in question is induced by an inner product. In such a case the corresponding isometries have no structure and there are plenty of them.

\section{Proof of Theorem~\ref{T:1}}

Let $\Gamma$ be any compact Lie group lying between $PSU(n)$ and $GL(n^2-1,\mathbb{R})$. By the structure theorem of compact connected Lie groups (see e.g. \cite[p. 241]{onishchik}), the connected one-component $\Gamma_0$ is a product of its commutator subgroup $\Gamma_0'$ and a torus which is a subgroup of the centralizer $Z(\Gamma_0)$. In particular, the commutator subgroup is semisimple.
As the adjoint action is (absolutely) irreducible on $H_n^0$, by Schur's lemma and the compactness of $\Gamma_0$, $Z(\Gamma_0)$ lies in both the scalars and $O(n^2-1,\mathbb{R})$. Since $O(n^2-1,\mathbb{R})\cap$ scalars is finite we conclude that $\Gamma_0$ is necessarily semisimple.

Next observe that $sl(n,\mathbb{C})=H_{n}^0\otimes \mathbb{C}$.
The group $PSU(n)$ acts also on $sl(n,\mathbb{C})$. The compact semisimple overgroups of $PSU(n)$ in $GL(n^2-1,\mathbb{C})$ were determined in \cite[Corollary 2.3]{guralnick}.
By that result we have the following list on compact semisimple overgroups:

\begin{center}
 \begin{tabular}{||c c c||}
 \hline
 $X(\mathbb{C})$ & $X$ & Comments \\ [0.5ex]
 \hline\hline
 $PSL(n,\mathbb{C})$ & $PSU(n)$ & $H_n^0$ invariant \\
 \hline
 $\Lambda(\mathbb{C})$ & $\Lambda$ & $n=4$   \\
 \hline
 $SL(n^2-1,\mathbb{C})$ & $SU(n^2-1)$ &  \\
 \hline
 $SO(n^2-1,\mathbb{C})$ & $SO(n^2-1,\mathbb{R})$ & $H_n^0$ invariant \\
 \hline
\end{tabular}
\end{center}
where $\Lambda$ is a maximal compact subgroup of $SU(6)/\langle -I \rangle$ and $X(\mathbb{C})$ stands for the Zariski closure (or complexification) of the Lie group $X$.
The group $\Lambda$ cannot leave the space $H_n^{0}$ invariant as this representation (the exterior square of the natural module, see \cite[Section 3]{PD} for its definition) is not defined over $\mathbb{R}$. Clearly, this holds true for $SU(n^2-1)$.
It follows that the compact semisimple overgroups in $GL(n^2-1,\mathbb{R})$ are $\Gamma_0=PSU(n),SO(n^2-1,\mathbb{R})$.

Let $N(\Gamma_0)$ denote the normalizer of $\Gamma_0$ in $GL(n^2-1,\mathbb{R})$. The inclusion
\[
\Gamma_0 \leq \Gamma \leq N(\Gamma) \leq N(\Gamma_0),
\]
is fundamental and gives us that any compact overgroup normalizes a semisimple one.
The normalizer of $SO(n^2-1,\mathbb{R})$ is $$\langle \mathbb{R}^{\times},SO(n^2-1,\mathbb{R})\rangle =\langle \mathbb{R}^{\times}_{+},O(n^2-1,\mathbb{R})\rangle$$
(see e.g. \cite[Theorem 2.5]{DL}).
Therefore, we just need to calculate the normalizer of $PSU(n)$. For the readers familiar with representation theory it is a simple matter but here we give a much elementary proof. Assume that $L$ lies in the normalizer. Then for any $U\in U(n)$ there exists a $V\in U(n)$ such that
\begin{equation} \label{normal}
    L(UXU^{-1})=VL(X)V^{-1}
\end{equation}
holds for all $X\in H_{n}^0$. Set $Y=L(X)$. It follows that for the unitary similarity orbit of $X$
\[
\tilde{\mathcal{U}}(X)=\{UXU^{-1}: \quad U\in U(n)\}
\]
we have $L(\tilde{\mathcal{U}}(X))\subseteq \tilde{\mathcal{U}}(Y)$. Substitute $L^{-1}$ and $L(X)$ in \eqref{normal} into the places of $L$ and $X$, respectively. We infer that for any $U\in U(n)$ there is a $V\in U(n)$ such that
\[
L^{-1}(UL(X)U^{-1})=VXV^{-1}
\]
and thus
\[
UYU^{-1}=UL(X)U^{-1}=L(VXV^{-1})
\]
satisfied. This yields $L(\tilde{\mathcal{U}}(X)) = \tilde{\mathcal{U}}(Y)$. We deduce that the unitary similarity orbit of an element mapped to another. Therefore, there is a map $\tilde{L}$ on the whole matrix algebra $M_n(\mathbb{C})$ such that $\tilde{L}(\tilde{\mathcal{U}}(X))=\tilde{\mathcal{U}}(Y)$ holds, and $L$ is the restriction of $\tilde{L}$ to the set of self-adjoint traceless matrices.
By part (ii) of \cite[Theorem 3.2]{orbit} we conclude that there exist $\eta\in \mathbb{R}, U \in U(n)$ and $B\in M_n(\mathbb{C})$ such that either $$\tilde{L}(X)=\eta UXU^{-1}+(\tr X)B$$ or $$\tilde{L}(X)=\eta U\sigma(X)U^{-1}+(\tr X)B$$ holds for all $X\in M_n(\mathbb{C})$.
Taking all the information what we have into account we conclude that the normalizer is a subgroup of $\langle \mathbb{R}^{\times}$, $PSU(n), \sigma \rangle$. Since
\[
\sigma(UAU^{-1})=(U')^{-1}\sigma(A)U' \qquad (A\in H_n^0)
\]
we deduce that the Cartan involution normalizes $PSU(n)$. Hence the normalizer is $\langle \mathbb{R}^{\times}$, $PSU(n), \sigma \rangle$.

An isometry group of a unitary similarity invariant norm is compact, and must contain $PSU(n),\mathbb{Z}_2$ and $\langle \sigma \rangle$. It follows that $\mathcal{K}$ is of the form as we have desired. The proof is completed.

\section{Additional results}

Next we list some direct consequences of Theorem~\ref{T:1}. We remark that questions of similar kind were answered in \cite{guralnick},\cite{iso}.
Recall that two normed spaces $(V_1,\|.\|_1)$ and $(V_2,\|.\|_2)$ are said to be isometric if there exists a linear mapping $L:V_1\to V_2$ such that $\|L(x)\|=\|x\|$ holds for all $x\in V_1$.

\begin{corollary}
Assume that $\|.\|_1$ and $\|.\|_2$ are norms on $H^0_{n}$ with isometry groups $\mathcal{K}_1$ and $\mathcal{K}_2$. Then an isometric isomorphism $L:H^0_{n}\to H^0_{n}$ exists if and only if $\mathcal{K}_1=\mathcal{K}_2$ and $\|.\|_1=\gamma \|.\|_2$ with some $\gamma > 0$ and $\gamma^{-1}L\in \mathcal{K}_1$.
\end{corollary}

Combining Theorem~\ref{T:1} with the result of Nagy \cite{nagy} on automatical linearity (up to a translation) of isometries on self-adjoint traceless matrices yields the following

\begin{corollary}
Suppose that $L$ is an isometry with respect to any unitary similarity invariant norm on $H^0_{n}$. Then there exists a unitary matrix $U\in U(n)$, $B\in H^0_{n}$ and a real number $\eta \in \{-1,1\}$ such that
\[
L(A)=\eta UAU^{-1}+B, \qquad A \in H^0_{n}
\]
or
\[
L(A)=\eta U\sigma(A)U^{-1}+B, \qquad A \in H^0_{n}.
\]
\end{corollary}

By using the result on compact overgroups of $PSU(n)$ (listed in the previous section) one can solve easily various linear preserver problems on the set of self-adjoint traceless matrices including certain generalized numerical ranges and numerical radii, as well. For a fixed $C\in H_n^0$ the $C$-numerical range and the $C$-numerical radius of an $A\in H_n^0$ is defined by
\[
W_C(A)=\{\tr AX:\quad X\in  \tilde{\mathcal{U}}(C)  \}
\]
and
\[
r_C(A)=\{|\lambda|: \quad \lambda \in W_C(A)\},
\]
respectively.

\begin{theorem}
Assume that $L$ is a linear map on $H_n^0$ satisfying either $W_C(L(A))=W_C(A)$ or $r_C(L(A))=r_C(A)$ for all $A\in H_n^0$. Then there exist $U\in U(n)$ and $\eta \in \{-1,1\}$ such that $L$ is one of the following forms:
\begin{itemize}
    \item[(i)] $L(A)=\eta UAU^{-1}, \quad A\in H_n^0$;
    \item[(ii)] $L(A)=\eta U\sigma(A)U^{-1}, \quad A\in H_n^0$.
\end{itemize}
\end{theorem}

\section{The case of skew-symmetric matrices}

Denote $K_n(\mathbb{F})$ the set of $n$ by $n$ skew-symmetric matrices over the field $\mathbb{F}$. Let $A':=\tau(A)$ be the transpose map and denote $A^*:=\psi(A)$ the operation on $K_4(\mathbb{F})$ defined by
\[
\begin{pmatrix}
0 & a_{12} & a_{13} & a_{14} \\
-a_{12} & 0 & a_{23} & a_{24} \\
-a_{13} & -a_{23} & 0 & a_{34} \\
-a_{14} & -a_{24} & -a_{34} & 0
\end{pmatrix}^*=\begin{pmatrix}
0 & a_{12} & a_{13} & a_{23} \\
-a_{12} & 0 & a_{14} & a_{24} \\
-a_{13} & -a_{14} & 0 & a_{34} \\
-a_{23} & -a_{24} & -a_{34} & 0
\end{pmatrix}.
\]
Let $PSO(n,\mathbb{R})$ be the image of $SO(n,\mathbb{R})$ under the adjoint representation
\[
Ad: SO(n,\mathbb{R})\to GL(K_n(\mathbb{R})), \quad Q\mapsto QAQ^{-1}=QAQ'.
\]
In \cite{chi} the following result was achieved and presented by Li and Tsing.

\begin{theorem}\label{CK}
Assume that $n\geq 3$. Let $\mathcal{K}$ be the isometry group of an orthogonal congruence invariant norm on $K_n(\mathbb{R})$ which is not a constant multiple of the Frobenius norm. Then one of the following holds:
\begin{itemize}
    \item [(i)] $\mathcal{K}=\langle PSO(n,\mathbb{R}), \tau \rangle$;
    \item [(ii)] $n=4$ and $\mathcal{K}=\langle PSO(n,\mathbb{R}), \tau, \psi \rangle$.
\end{itemize}
\end{theorem}

As far as the authors know, there is not any complete proof available in the literature. By using the proof techniques appearing in the previous section we could present a short one. Since the machinery is rather close in spirit to that of Section 2, the interested reader should be able to fill in the details.

Set $m=n(n-1)/2$, that is, the dimension of $K_n(\mathbb{F})$ and introduce the following groups:
\begin{itemize}
    \item [(i)] $SO(m,\mathbb{R})$: the special orthogonal group on $K_n(\mathbb{R})$ (with respect to the trace form)
    \item [(ii)] $SO(m,\mathbb{C})$: the special orthogonal group on $K_n(\mathbb{C})$
    \item [(iii)] $SL(m,\mathbb{C})$: the special linear group on $K_n(\mathbb{C})$
    \item [(iv)] $SU(m)$: the special unitary group on $K_n(\mathbb{C})$
\end{itemize}
We note that if $n$ is odd, then $PSO(n,\mathbb{R})$ is isomorphic to $SO(n,\mathbb{R})$. Otherwise, if $n$ is even, $PSO(n,\mathbb{R})$ is isomorphic to $SO(n,\mathbb{R})/\langle -I \rangle$. Furthermore, the groups $PSO(n,\mathbb{R})$, $SO(m,\mathbb{R})$ and $SU(m)$ have Zariski closures $PSO(n,\mathbb{C})$, $SO(m,\mathbb{C})$ and $SL(m,\mathbb{C})$, respectively.

\begin{theorem}\label{T:2}
If $\Gamma$ is a compact semisimple subgroup of $GL(K_n(\mathbb{R}))$ which contains $PSO(n,\mathbb{R})$, then $\Gamma=PSO(n,\mathbb{R})$ or $\Gamma=SO(K_n(\mathbb{R}))$.
\end{theorem}

In the proof we shall apply the following classification result of Dynkin on contains of irreducible subgroups of $SL(N)$.

\begin{theorem}\cite[Theorem 2.3]{dynkin} \label{T:D}
Table 5 of \cite{dynkin} gives a complete list of all inclusion types $G^*\subsetneq G \subsetneq SL(N,\mathbb{C})$ such that
\begin{itemize}
\item[(i)] $G^*,G$ are distinct irreducible Lie subgroups;
\item[(ii)] $G^*$ is connected and $G$ is simple;
\item[(iii)] $G$ is not conjugate in $GL(N,\mathbb{C})$ to any of $SL(N,\mathbb{C})$, $SO(N,\mathbb{C})$ or $Sp(N,\mathbb{C})$.
\end{itemize}
\end{theorem}

Note that our applications involve only the case when $G^*$ is simple, too.

\begin{remark}
By the term inclusion type we mean that two pairs of subgroups $A^* \subsetneq A$ and $B^* \subsetneq B$ are in the same inclusion type whenever $A^*$ and $A$ are conjugate to $B^*$ and $B$ (in $GL(N,\mathbb{C})$, via the same conjugation).
\end{remark}

\begin{proof}[Proof of Theorem~\ref{T:2}]
By \cite[Theorem 19.2, 19.14]{FuHa} $K_n(\mathbb{C})$ is an irreducible $PSO(n,\mathbb{C})$-module.

First assume that $n>4$.
If $X(\mathbb{C})$ is a complex semisimple subgroup of $GL(m,\mathbb{C})$ such that $PSO(n,\mathbb{C})\leq X(\mathbb{C}) \leq GL(m,\mathbb{C})$ holds, then $X(\mathbb{C})$ is simple ($K_n(\mathbb{C})$ is an irreducible, tensor indecomposable module because it is already so under $PSO(n,\mathbb{C})$).
By Theorem~\ref{T:D}, it follows that if $X(\mathbb{C})\neq PSO(n,\mathbb{C})$, then either $X(\mathbb{C})=SL(m,\mathbb{C}),SO(m,\mathbb{C})$ or $X(\mathbb{C})$ is given explicitly in \cite[Table 5]{dynkin}.
Therefore, if $n>4$, then we have the following list on the corresponding overgroups:
\begin{center}
 \begin{tabular}{||c c c||}
 \hline
 $X(\mathbb{C})$ & $X$ & $K_n(\mathbb{R})$ invariant \\ [0.5ex]
 \hline\hline
 $PSO(n,\mathbb{C})$ & $PSO(n,\mathbb{R})$ & yes \\
 \hline
 $SL(n,\mathbb{C})/\langle -I\rangle$ & $SU(n)/\langle -I\rangle$ &  no  \\
 \hline
 $SL(m,\mathbb{C})$ & $SU(m)$ & no \\
 \hline
 $SO(m,\mathbb{C})$ & $SO(m,\mathbb{R})$ & yes \\
 \hline
\end{tabular}
\end{center}

The case $n=4$ is a bit different because $K_4(\mathbb{R})$ is not a simple module for $PSO(4,\mathbb{C})$. Then either $X(\mathbb{C})=SL(6,\mathbb{C}),SO(6,\mathbb{C})\cong SL(4,\mathbb{C})/\langle -I\rangle$ or $X(\mathbb{C})$ is semisimple but not simple.
In this latter case $X(\mathbb{C})$ is a nontrivial product of simple Lie groups. According to part (d) of \cite[Theorem 2.2]{G2} (c.f. \cite[Theorem 1.3]{dynkin}) this can happen only when $X(\mathbb{C})=PGL(2,\mathbb{C})\times PGL(2,\mathbb{C})\cong PSO(4,\mathbb{C})$.

For $n=m=3$ we note that $PSO(3,\mathbb{R})$ is isomorphic to $SO(3,\mathbb{R})$.

Hence we conclude that the semisimple overgroups are the same as we have desired.
\end{proof}

Now we are in a position to present the proof of our last result. We will use the following notation in it. The symbol $\lambda$ stands for the highest weight with respect to any representation of a semisimple algebraic group $G$ on the vector space $V$. Let $Aut(\Delta)$ be the automorphisms group of the Dynkin diagram $\Delta$. Denote $Aut(\Delta, \lambda)$ the subgroup of $Aut(\Delta)$ fixing $\lambda$. Furthermore, we specify $G:=PSO(n,\mathbb{R})$ and $V:=K_n(\mathbb{R})$.

\begin{proof}[Proof of Theorem~\ref{CK}]

In a similar fashion as in Section 2. we conclude that any overgroup of $G$ normalizes a semisimple one. As for the normalizer of $G$ in $GL(V)$, we can consider the Youla decomposition of an $A\in V$ as $A=Q\Sigma Q'$ where
\[
\Sigma=\textbf{0}_{n-2r} \bigoplus_{i=1}^{r} \begin{pmatrix}
0 & a_i \\
-a_i &  0
\end{pmatrix}
\]
with singular values $a_1 \geq a_2 \geq ... \geq a_r > 0$.
We observe that the orthogonal congruence orbit of $A$, i.e.,
$$\mathcal{O}(A)=\{QAQ': \quad Q\in O(n,\mathbb{R}) \}$$
consists of skew-symmetric matrices with fixed singular values. Therefore, \cite[Theorem 4.2]{chi} is useful for the description of mappings which sends an orthogonal congruence orbit to another. Hence we conclude that if $n\neq 4$, then
\[
N_{GL(V)}(G)\leq \langle \mathbb{R}_+^\times , G, \tau \rangle = \langle \mathbb{R}^\times , G \rangle
\]
and thus the normalizer of $G$ is $\langle \mathbb{R}_+^\times , G, \tau \rangle$. If $n=4$,
the normalizer is a subgroup of $N'=\langle \mathbb{R}^\times , G, \psi \rangle$, again by \cite[Theorem 4.2]{chi}. On the other hand, writing $\pi_0$ to mean the component group, we have 
\[
\pi_0(N_{GL(V)}(G)) \cong Aut(A_1\times A_1,\lambda) = \mathbb{Z}/2\mathbb{Z} \neq 1.
\]
Therefore, $\psi \in N_{GL(V)}(G)$ and it follows that $N_{GL(V)}(G)$ equals to $N'$.
Since $A$ and $\psi(A)$ have the same characteristic polynomial the result on isometry groups follows.
\end{proof}

We close this paper by noting that the last part of the above proof can be completed, alternatively, having proved the "upper bound" for $N_{GL(V)}(G)$ only. Indeed, referring to the fact that $\psi$ is an isometry the inclusions
\[
\mathcal{K}\leq N_{GL(V)}(G) \leq N'=\langle \mathbb{R}^\times , G, \psi \rangle
\]
yield the assertion.

\section{Acknowledgement}
The first author was partially supported by the National Research, Development and Innovation Office -- NKFIH Reg. No. K115383.
The second author was partially supported by NSF grants DMS-1302886 and DMS-1600056.

The first author would like to express his hearty thanks to Prof. G\'abor P\'eter Nagy for illuminating discussions.


\bibliographystyle{amsplain}

\end{document}